\documentclass{amsart}
\usepackage{amsmath,amssymb,bm}
\usepackage[utf8]{inputenc}
\usepackage{graphicx}
\usepackage{xcolor}

\newtheorem{theorem}{Theorem}
\newtheorem{lemma}{Lemma}

\newtheorem{remark}{Remark}
\newtheorem{proposition}{Proposition}

\title[Localization and the landscape function]{Localization and the landscape function for regular Sturm-Liouville operators}
\author{Mirza Karamehmedovi\'c and Faouzi Triki}
\date{}

\begin{document}

\begin{abstract}
We consider the localization in the eigenfunctions of regular Sturm-Liouville operators. After deriving non-asymptotic and asymptotic lower and upper bounds on the localization coefficient of the eigenfunctions, we
characterize the landscape function in terms of the first eigenfunction. 
Several numerical experiments are provided to illustrate the obtained theoretical results.
\end{abstract}

\maketitle

\section{Introduction and main results}\label{sec:intro}

Let $L>0$, assume $p,w\in C^2([0,L])$ are positive-valued functions satisfying $p^{-1},w^{-1}\in L^\infty(]0,L[)$, and let $q\in C([0,L])$ be nonnegative-valued. 
Define the unbounded operator $T$  by
\[
Tu=-\frac{1}{w}(pu')'+\frac{q}{w}u,
\]
with domain 
\[
D(T) = \left\{ u \in L^2(]0,L[,w(x)dx): \; Tu \in L^2(]0,L[,w(x)dx), \; u(0)=u(L)=0 \right\}, 
\]
and recall that $T$ is self-adjoint in $L^2(]0,L[,w(x)dx)$ with a compact resolvent. We here investigate the "localization" in the solution $\phi_{\lambda}\in D(T)$ of the regular Sturm-Liouville problem
\begin{equation}\label{eqn:MAIN}
\begin{array}{rcl}T\phi_{\lambda}&=&\lambda\phi_{\lambda},
\end{array}
\end{equation}
for positive $\lambda$. In particular, writing $\|\cdot\|_t$ for $\|\cdot\|_{L^t(]0,L[)}$, we find \textit{non-asymptotic as well as asymptotic} lower and upper bounds for the 'existence surface'~\cite{Felix-2007,filoche2012localization}, also called the 'localization coefficient',
\[
\alpha(\phi_{\lambda})=\|\phi_{\lambda}\|^4_2/\|\phi_{\lambda}\|^4_4.
\]
The quantity $\alpha(\phi_{\lambda})$ is independent of any normalization of $\phi_{\lambda}$ by a scalar factor, and it is a standard measure of the localization of $\phi_{\lambda}$, with low $\alpha(\phi_{\lambda})$ indicating high localization. 'High localization' means that the amplitude of the solution function is relatively high over a small connected sub-interval $I\subset ]0,L[ $, and relatively low in $]0,L[ \setminus I$. Figure~\ref{fig:loc} helps illustrate the concept of localization. Here, we let $L=1$, $q\equiv0$, $w\equiv1$, and 
\begin{equation}\label{eqn:p}
p(x)=\tanh(40 x/L - 10)+1.1,\quad x\in[0,L],
\end{equation}
making the operator $T$ in~\eqref{eqn:MAIN} the Dirichlet Laplacian on $[0,1]$ with a non-trivial metric, $Tu=-(pu')'$.
\begin{figure}
\begin{center}
\includegraphics[width=0.49\linewidth]{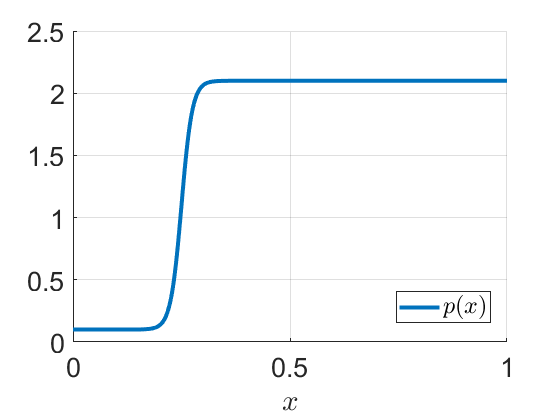}\includegraphics[width=0.49\linewidth]{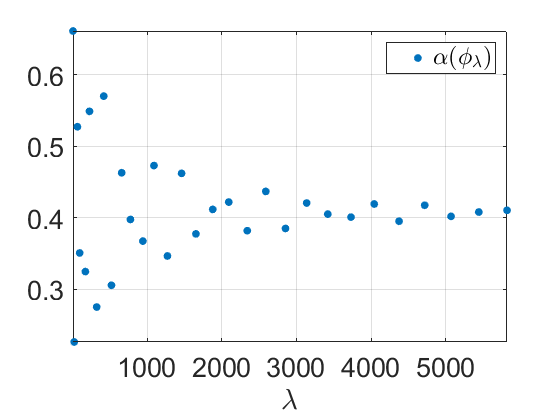}\\
\includegraphics[width=0.49\linewidth]{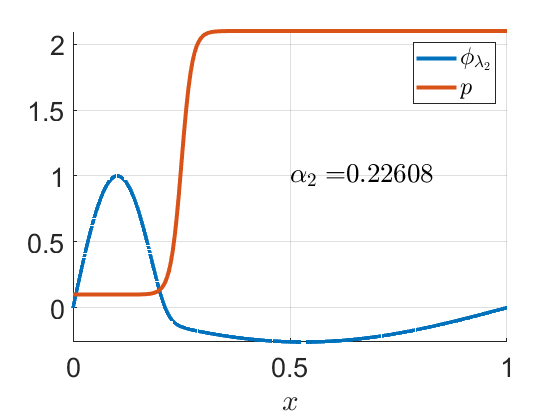}\includegraphics[width=0.49\linewidth]{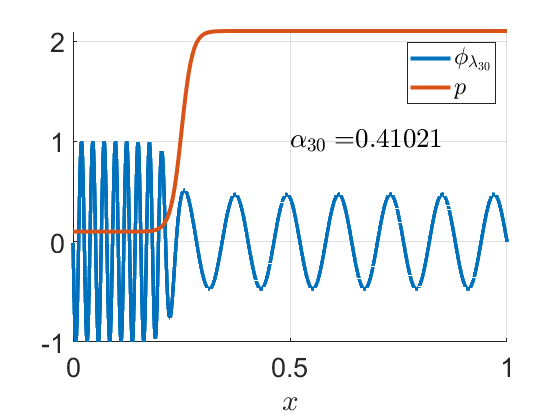}
\caption{Localization of eigenfunctions of the Dirichlet Laplacian on $[0,1]$ with metric $p$.}\label{fig:loc}
\end{center}
\end{figure}
Note that the most localized eigenfunctions correspond to relatively small eigenvalues, and that the localization coefficient seems to approach a constant with increasing eigenvalues. Our results, valid well beyond this single example case, predict both these empirical observations on localization.\\

We first derive lower and upper bounds for $\alpha(\phi_{\lambda})$ in the non-asymptotic regime, specifically showing that $\alpha(\phi_{\lambda})$ can 
attain relatively low values only at relatively low frequencies (small $\lambda$). Then, to complete the picture, we prove the lower and upper bounds for $\alpha(\phi_{\lambda})$ in the asymptotic regime as $\lambda \to \infty$.\\

The treatment of the Sturm-Liouville problem~\eqref{eqn:MAIN} when the coefficients are smooth usually starts with the Liouville transformation to the eigenvalue problem for the Schr\"odinger operator \cite{bao2020stability}. We work with this transformation in Section~\ref{sec:proof1}, but to state our second and third main results we already here define some of the involved quantities. Thus let
\[
y(x)=\int_0^x\sqrt{w(s)/p(s)}ds,\quad x\in[0,L],
\]
and
\[
B=\int_{s=0}^L\sqrt{w(s)/p(s)}ds.
\]
The function $y: ]0, L[ \to ]0, B[$ is  strictly increasing, and has an inverse denoted by $x(y)$. Let
\[
f(y)=( w(y))p(x(y)))^{1/4},\quad y\in[0,B],
\]
and 
\begin{equation}\label{eqn:Q}
Q(y)=f''(y)/f(y) + q(x(y))/w(x(y)),\quad y\in[0,B].
\end{equation}
Write also
\begin{equation}\label{eqn:a}
a(B,\lambda)=\frac{B\|Q\|_{\infty}}{2\sqrt{\lambda}},
\end{equation}
and
\begin{equation}\label{eqn:b} b(B,\lambda)=\left(\frac{B^3}{12}+\frac{5B}{32\lambda}+\frac{5}{32\lambda^{3/2}}\right)^{1/4}\|Q\|_4/\sqrt{\lambda}.
\end{equation}
Finally, for any real $\lambda$, let
\[
\Phi_{\lambda}(y)=\sin(\sqrt{\lambda}y),\quad y\in[0,B].
\]
Our first main result gives non-asymptotic bounds on $\alpha(\phi_{\lambda})$. Let
\[
\beta(p,w)=\|w\|_{\infty}^{-2}\|p^{-1/2}w^{-3/2}\|_{\infty}^{-1}\quad{\rm and}\quad\gamma(p,w)=\|w^{-1}\|_{\infty}^2\|p^{1/2}w^{3/2}\|_{\infty}.
\] 
\begin{theorem}\label{thm:non-asymptotic}
If
\begin{equation}\label{eqn:assum1}
a(B,\lambda)<1 
\end{equation}
and
\begin{equation}\label{eqn:assum2}
b(B,\lambda)<1 
\end{equation}
then
\[
\beta(p,w)\left(\frac{1-b(B,\lambda)}{1+a(B,\lambda)}\right)^4\le\frac{\alpha(\phi_{\lambda})}{\alpha(\Phi_{\lambda})}\le\gamma(p,w)\left(\frac{1+b(B,\lambda)}{1-a(B,\lambda)}\right)^4.
\]
\end{theorem}
Figure~\ref{fig:bounds} illustrates the bounds on the localization coefficient from Theorem~\ref{thm:non-asymptotic}.
\begin{figure}
\includegraphics[scale=0.4]{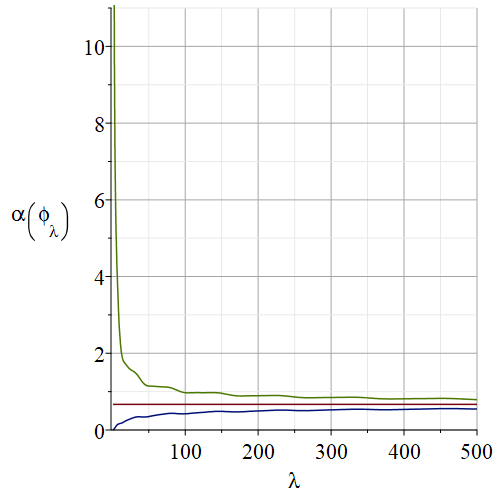}
\caption{Lower and upper bounds on $\alpha(\phi_{\lambda})$ from Theorem~\ref{thm:non-asymptotic} for the eigenvalue problem~\eqref{eqn:MAIN} after Liouville transformation, with $\beta(p,w)=1$, $\gamma(p,w)=1$, $B=1$, $\|Q\|_{\infty}=1$, and $\|Q\|_4=1$. The constant value is the asymptotic $2B/3$. For the chosen parameter values, the assumptions~\eqref{eqn:assum1}--\eqref{eqn:assum2} are satisfied for $\lambda\gtrapprox0.74$.}\label{fig:bounds}
\end{figure}
Let ${\rm BV}([0,B])$  and ${\rm AC}([0,B])$ be respectively  the space of  bounded variation functions, and 
the space of absolutely continuous functions. \\

Our final main result concerns the asymptotic behavior of $\alpha(\phi_{\lambda})$:
\begin{theorem}\label{thm:asymptotic} As $\lambda\rightarrow\infty$, we have
\[
\beta(p,w)\frac{2B}{3}+O(\lambda^{-1/2})\le\alpha(\phi_{\lambda})\le\gamma(p,w)\frac{2B}{3}+O(\lambda^{-1/2})
\]
when $Q\in C([0,B])$,
\begin{align*}
\beta(p,w)\frac{\frac{B^2}{4}-B(\frac{1}{4}+\|Q\|_1B)\lambda^{-1/2}}{\frac{3B}{8}+(\frac{9}{32}+2\|Q\|_1B)\lambda^{-1/2}}&+O(\lambda^{-1})\le\alpha(\phi_{\lambda})\\&\le\gamma(p,w)\frac{\frac{B^2}{4}+B(\frac{1}{4}+\|Q\|_1B)\lambda^{-1/2}}{\frac{3B}{8}-(\frac{9}{32}+2\|Q\|_1B)\lambda^{-1/2}}+O(\lambda^{-1})
\end{align*}
when $Q\in{\rm BV}([0,B])$, and
\[
\beta(p,w)\frac{2B}{3}+O(\lambda^{-3/2})\le\alpha(\phi_{\lambda})\le\gamma(p,w)\frac{2B}{3}+O(\lambda^{-3/2})
\]
when $Q\in C^4([0,B])\cap{\rm AC}([0,B])$ and $Q'\in{\rm BV}([0,B])$.
\end{theorem}
\begin{remark}
A straightforward calculation shows that the localization coefficient of the function $\Phi_{\lambda}$ is for any positive $\lambda$ given by
\begin{equation}\label{eqn:alpha_f}
\alpha(\Phi_{\lambda})=\frac{B^2/4+\cos(\sqrt{\lambda}B)^2/4\lambda-B\cos(\sqrt{\lambda}B)\sin(\sqrt{\lambda}B)/2\sqrt{\lambda}-\cos(\sqrt{\lambda}B)^4/4\lambda}{3B/8+\cos(\sqrt{\lambda}B)^3\sin(\sqrt{\lambda}B)/4\sqrt{\lambda}-5\cos(\sqrt{\lambda}B)\sin(\sqrt{\lambda}B)/8\sqrt{\lambda}}.
\end{equation}
The eigenvalues and eigenfunctions of the Dirichlet Laplacian $-d^2/dy^2$ on $(0,B)$ are given by $\lambda_n=n^2\pi^2/B^2$ and $\Phi_{\lambda_n}(y)=\sin(n\pi y/B)$, respectively, with $n\in\bm{N}_0 := \bm{N}\setminus \{0\}=\{1,2,\dots\}$. In particular, $\alpha(\Phi_{\lambda_n})=2B/3$ for $n\in\bm{N}_0$, that is, all eigenfunctions of the Dirichlet Laplacian on $(0,B)$ have the same localization coefficient. We furthermore readily see that
\begin{equation}\label{eqn:flim}
\lim_{\lambda\rightarrow\infty}\alpha(\Phi_{\lambda})=\frac{2B}{3},
\end{equation}
and more precisely that, for large $\lambda$,
\begin{align}\label{eqn:bvc}
\alpha(\Phi_{\lambda})&=\frac{B^2/4+O(\lambda^{-1/2})}{3B/8+O(\lambda^{-1/2})}=\frac{2B}{3}\frac{1}{1+O(\lambda^{-1/2})}+O(\lambda^{-1/2})\nonumber\\&=\frac{2B}{3}+O(\lambda^{-1/2}).
\end{align}
Figure~\ref{fig:alpha_Phi} shows $\alpha(\Phi_{\lambda})$ as function of $\lambda$ for the choice $B=1$.
\begin{figure}
\includegraphics[scale=0.4]{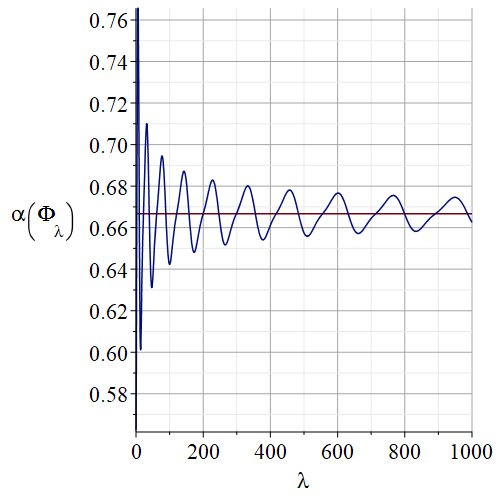}
\caption{The localization coefficient $\alpha(\Phi_{\lambda})$ of the function $\Phi_{\lambda}(y)=\sin\sqrt{\lambda}y$, $y\in[0,B]$, for $B=1$. The constant value shown is $\alpha(\Phi_{\lambda_n})=2/3$, $n\in\bm{N}_0$, where the $\Phi_{\lambda_n}$ are the eigenfunctions of the Dirichlet Laplacian on the interval $[0,1]$.}\label{fig:alpha_Phi}
\end{figure}
\end{remark}
\begin{remark}
If $p\equiv1$ and $w\equiv1$ then $T$ is the Schr\"odinger operator $-d^2/dx^2+q(x)$. For this case, in the high-frequency limit ($\lambda\rightarrow\infty$) the localization coefficient of $\phi_{\lambda}$ approaches that of $\Phi_{\lambda}$ (so it approaches the value $2B/3$), that is, \textbf{the presence of the potential $q$ becomes insignificant}.
\end{remark}
\begin{remark}
It is readily seen that the lower bound on $\alpha(\phi_{\lambda})/\alpha(\Phi_{\lambda})$ in Theorem~\ref{thm:non-asymptotic} is a monotonically increasing function of $\lambda$, for any fixed positive $B$. In view of this, and of the behavior of $\alpha(\Phi_{\lambda})$ discussed above, we conclude that if a solution of~\eqref{eqn:MAIN} is to exhibit high localization (relatively small value of $\alpha(\phi_{\lambda})$) then $\lambda$ must be relatively small, that is, \textbf{localization is a low-frequency phenomenon}.
\end{remark}

We next focus on the localization in the eigenfunctions associated to low frequencies  (small $\lambda$). In~\cite{filoche2012universal} the authors
have given a simple but efficient way to predict the behavior of first eigenfunctions. Precisely they used the 
 {\it landscape function} $\ell\in D(T)$, solving
\begin{equation}
T\ell=1, \quad  \ell\in D(T),
\end{equation}
to identify the regions where the solution of~\eqref{eqn:MAIN} localizes. This can be observed through  the following  pointwise key inequality \cite{moler1968bounds}
\begin{equation*}
\phi(x) \leq \lambda \ell(x) \|\phi\|_\infty, \quad x\in \Omega.
\end{equation*}
Indeed $\lambda \|\ell\|_\infty \geq 1$ and  
$\phi$ can then localize only in the region $\{ x\in \Omega;  \lambda \ell(x) \geq 1\}$.\\ 

Our first main result in this part of the paper thus characterizes the landscape function  in terms of the first eigenfunction of the operator $T$. 
\begin{proposition}\label{prop:landscape-function}
Assume that $w\equiv 1$, and  let $k\in \bm{N}_0 $. If
\[T^k\widetilde{\ell}_k=1,\quad \widetilde{\ell}_k\in D(T^k),\quad \ell_k=\widetilde{\ell}_k/\|\widetilde{\ell}_k\|_2,\] as well as \[T\phi_1=\lambda_1\phi_1,\quad\phi_1\in D(T),\quad \|
\phi_1\|_{2}=1,\quad \phi_1>0, \quad \lambda_1<\lambda_j\,\,\,{\rm for}\,\,\,j=2,3,\dots,\]
where $\lambda_j, j\in \bm{N}_0 $ is the non-decreasing  sequence of eigenvalues of $T$. Then 
\begin{equation} \label{iterated-landscape}
\|\ell_k-\phi_1\|_\infty \leq  2\lambda_1^{1/2}L\|P_11\|_2^{-1}\left(\frac{\lambda_1}{\lambda_2}\right)^{k-1/2},\end{equation}
where $P_1$ is the spectral projection onto the eigenspace associated with $\lambda_1$.

\end{proposition}
\begin{remark} 
The asymptotic result in~\eqref{iterated-landscape} shows that the
convergence is exponentially fast if the {\it fundamental gap} $\lambda_2 -\lambda_1$ is large
enough. When $p=1$, and $q$ is a weakly convex potential it is known  that \cite{lavine1994eigenvalue}
$$
\lambda_2-\lambda_1 \geq \frac{3\pi^2}{L^2}.
$$
The inequality conjectured by Yau~\cite{yau1993open} is still an open problem in higher dimensions.
It turns out that this spectral gap also determines the rate at which 
positive solutions of the heat equation tend to their projections onto the  first eigenspace.
\end{remark}
\begin{proposition} \label{prop: iterated landscape function} Assume that $w\equiv 1$, and let $\lambda_j$, $j\in \bm{N}_0$, be the non-decreasing  sequence of eigenvalues of $T$. Let
$P_j$ be the spectral projection onto the eigenspace associated with $\lambda_j$. 
    Let $k$, $n_0\in \bm{N}_0$, and $t\in ]\lambda_{n_0+1}^{-1}, \lambda_{n_0}^{-1}[$. If 
\[(tT)^k\ell_{k,t}=1,\quad \ell_{k,t}\in D(T^k),\]   then 
\begin{equation} \label{iterated landscape large index}
\|\ell_{k,t}-\sum_{j=1}^{n_0}\frac{1}{(t\lambda_j)^k}P_j1 \|_\infty  \leq \frac{L}{t^{1/2}}\frac{1}{(t\lambda_{n_0+1})^{k-1/2}}.\end{equation}
\end{proposition}
\begin{remark} The value of $t$  fix the number of the eigenfunctions covered by the generalized landscape function $\ell_{t,k}$. Notice that 
$t\in ]\lambda_{n_0+1}^{-1}, \lambda_{n_0}^{-1}[$ is equivalent to
\begin{equation*}
\frac{1}{t\lambda_{n_0+1}} <1 <\frac{1}{t\lambda_{n_0}},
\end{equation*}
which implies that the contribution of the eigenfunctions $P_j1, \; j> n_0$ in  the localization of $\ell_{k, t}$ is exponentially small for large $k$ while the contribution
of $P_j1$ for $ 1\leq j\leq  n_0$ can be exponentially large if in addition $P_j1$ is not zero. These observations are confirmed in Section \ref{sec:numerical tests} by several numerical tests. Finally the results of Propositions  \ref{prop:landscape-function} and  \ref{prop: iterated landscape function}  are still valid for  $w$ non-constant and sufficiently smooth ($\|\cdot\|_2$ should be substituted by  $\|\cdot\|_{L^2(]0, L[; wdx)})$).  
\end{remark}


Theorem~\ref{thm:non-asymptotic} is proved in  Section~\ref{sec:proof1} using a Volterra integral equation representation of solutions of~\eqref{eqn:MAIN} given by Fulton~\cite{Fulton-1982}, while  Theorem~\ref{thm:asymptotic} is proved in Section~\ref{sec:asymptotic} via the asymptotic expansions of $\phi_{\lambda}$, as $\lambda\rightarrow\infty$, given in Fulton and Pruess~\cite{Fulton-Pruess-1994}. Finally,  we prove Propositions~\ref{prop:landscape-function} and \ref{prop: iterated landscape function} in respectively Sections~\ref{sec:proof2} and 
\ref{sec:proof2bis}
using the power method.

\section{Proof of Theorem~\ref{thm:non-asymptotic} (non-asymptotic bounds on $\alpha(\phi_{\lambda})$)}\label{sec:proof1}
Using the Liouville transformation
\[
y(x)=\int_0^x\sqrt{w(s)/p(s)}ds\,\,\,\text{for}\,\, x\in[0,L];\quad B=y(L);
\]
\[
f(y)=(w(x(y))p(x(y)))^{1/4},\quad y\in[0,B]; 
\]
\begin{equation}\label{eqn:v}
v_{\lambda}(y)=\phi_{\lambda}(x(y))f(y),\quad y\in[0,B];
\end{equation}
and
\[
\quad Q(y)=f''(y)/f(y) + q(x(y))/w(x(y)),\quad y\in[0,B],
\]
we recast the problem~\eqref{eqn:MAIN} in the Liouville normal form~\cite[pp. 303--304]{Fulton-Pruess-1994}
\begin{equation}\label{eqn:SV}
\left\{\begin{array}{rcl}
-v_{\lambda}''+Q(y)v_{\lambda}&=&\lambda v_{\lambda},\quad y\in(0,B),\\
v_{\lambda}(0)=v_{\lambda}(B)&=&0.
\end{array}\right.
\end{equation}
It follows from our assumptions on $p$, $w$ and $q$ that $Q\in C([0,B])$, hence $Q\in L^t([0,B])$ for $t\in[1,\infty]$.
Now
\[
\int_{x=0}^{L}\phi_{\lambda}(x)^2dx=\int_{y=0}^B\frac{v_{\lambda}(y)^2}{w(x(y))}dy\in\left[\|w\|^{-1}_{\infty},\|w^{-1}\|_{\infty}\right]\times\int_{y=0}^Bv_{\lambda}(y)^2dy
\]
and
\begin{align*}
\int_{x=0}^{L}\phi_{\lambda}(x)^4dx&=\int_{y=0}^B\frac{v_{\lambda}(y)^4}{p(x(y))^{1/2}w(x(y))^{3/2}}dy\\&\in\left[\|p^{1/2}w^{3/2}\|^{-1}_{\infty},\|p^{-1/2}w^{-3/2}\|_{\infty}\right]\times\int_{y=0}^Bv_{\lambda}(y)^4dy,
\end{align*}
so it remains to examine $\|v_{\lambda}\|_2^2$ and $\|v_{\lambda}\|_4^4$. To this end we recall from Fulton and Pruess~\cite[p. 308]{Fulton-Pruess-1994} that a solution of the ODE in~\eqref{eqn:SV}, normalized such that $v_{\lambda}(0)=0$ and $v'_{\lambda}(0)=(w(0)p(0))^{-1/4}\neq0$, satisfies the associated Volterra integral equation
\begin{equation}\label{eqn:Volterra}
\left({\rm Id}-\frac{1}{\sqrt{\lambda}}K_Q\right)v_{\lambda}(y)=\frac{v'_{\lambda}(0)}{\sqrt{\lambda}}\Phi_{\lambda}(y),\quad y\in[0,B],
\end{equation}
where for any $u\in C^2([0,B])$ we have
\[
K_Qu(y)=\int_{z=0}^yQ(z)\sin(\sqrt{\lambda}(y-z))u(z)dz,\quad y\in]0,B[.
\]
Now write $\|K_Q\|_t$ for the operator norm of $K_Q$ as a mapping from $L^t(]0,B[)$ to $L^t(]0,B[)$.
\begin{lemma}\label{lemma:norms}
For every positive $\lambda$ we have
\[
\|K_Q\|^2_2\le\frac{B^2}{4}\|Q\|^2_{\infty}
\]
and
\[
\|K_Q\|^4_4\le\left(\frac{B^3}{12}+\frac{5B}{32\lambda}+\frac{5}{32\lambda^{3/2}}\right)\|Q\|^4_4.
\]
\end{lemma}
\begin{proof}
The estimates follow readily from applying H\"older's inequality. We have
\begin{align*}
\|K_Qu\|^2_2&\le\|Q\|_{\infty}^2\|u\|_2^2\int_{y=0}^B\|\sin(\sqrt{\lambda}(y-\cdot))\|_{L^2(]0,y[)}^2dy\\&=\|Q\|_{\infty}^2\|u\|_2^2\int_{y=0}^B\left(\frac{y}{2}-\frac{\cos\sqrt{\lambda}y\sin\sqrt{\lambda}y}{2\sqrt{\lambda}}\right)\\&=\frac{\|Q\|_{\infty}^2}{4}\|u\|_2^2\left(B^2-\frac{\sin^2\sqrt{\lambda}B}{\lambda}\right)\\&\le\frac{\|Q\|_{\infty}^2B^2}{4}\|u\|_2^2,\quad u\in L^2(]0,B[),
\end{align*}
as well as
\begin{align*}
\|K_Qu\|^4_4&\le\|Q\|_4^4\|u\|_4^4\int_{y=0}^B\|\sin(\sqrt{\lambda}(y-\cdot))\|_{L^2(]0,y[)}^4dy\\&=\|Q\|_4^4\|u\|_4^4\left(\frac{B^3}{12}+\frac{B\cos(\sqrt{\lambda}B)^2}{4\lambda}-\frac{\sin\sqrt{\lambda}B\cos(\sqrt{\lambda}B)^3}{16\lambda^{3/2}}\right.\\&\left.-\frac{3B}{32\lambda}-\frac{3\cos\sqrt{\lambda}B\sin\sqrt{\lambda}B}{32\lambda^{3/2}}\right)\\&\le\|Q\|_4^4\left(\frac{B^3}{12}+\frac{5B}{32\lambda}+\frac{5}{32\lambda^{3/2}}\right)\|u\|_4^4,\quad u\in L^4([0,B]),
\end{align*}

\end{proof}
We have from~\eqref{eqn:Volterra} and from Lemma~\ref{lemma:norms} that, for all positive $\lambda$,
\begin{align*}
\|v_{\lambda}\|_2&\ge\lambda^{-1/2}|v'_{\lambda}(0)|\|\Phi_{\lambda}\|_2-\lambda^{-1/2}\|K_Q\|_2\|v_{\lambda}\|_2\\&\ge\lambda^{-1/2}|v'_{\lambda}(0)|\|\Phi_{\lambda}\|_2-\lambda^{-1/2}B\|Q\|_{\infty}\|v_{\lambda}\|_2/2,
\end{align*}
\begin{align*}
\|v_{\lambda}\|_2&\le\lambda^{-1/2}|v'_{\lambda}(0)|\|\Phi_{\lambda}\|_2+\lambda^{-1/2}\|K_Q\|_2\|v_{\lambda}\|_2\\&\le\lambda^{-1/2}|v'_{\lambda}(0)|\|\Phi_{\lambda}\|_2+\lambda^{-1/2}B\|Q\|_{\infty}\|v_{\lambda}\|_2/2,
\end{align*}
\begin{align*}
\|v_{\lambda}\|_4&\ge\lambda^{-1/2}|v'_{\lambda}(0)|\|\Phi_{\lambda}\|_4-\lambda^{-1/2}\|K_Q\|_4\|v_{\lambda}\|_4\\&\ge\lambda^{-1/2}|v'_{\lambda}(0)|\|\Phi_{\lambda}\|_4-\lambda^{-1/2}\left(\frac{B^3}{12}+\frac{5B}{32\lambda}+\frac{5}{32\lambda^{3/2}}\right)^{1/4}\|Q\|_4\|v_{\lambda}\|_4,
\end{align*}
and
\begin{align*}
\|v_{\lambda}\|_4&\le\lambda^{-1/2}|v'_{\lambda}(0)|\|\Phi_{\lambda}\|_4+\lambda^{-1/2}\|K_Q\|_4\|v_{\lambda}\|_4\\&\le\lambda^{-1/2}|v'_{\lambda}(0)|\|\Phi_{\lambda}\|_4+\lambda^{-1/2}\left(\frac{B^3}{12}+\frac{5B}{32\lambda}+\frac{5}{32\lambda^{3/2}}\right)^{1/4}\|Q\|_4\|v_{\lambda}\|_4.
\end{align*}
Specifically, for $B\|Q\|_{\infty}/2<\sqrt{\lambda}$ (assumption~\eqref{eqn:assum1}) and
\[
\left(\frac{B^3}{12}+\frac{5B}{32\lambda}+\frac{5}{32\lambda^{3/2}}\right)^{1/4}\|Q\|_4<\sqrt{\lambda}
\]
(assumption~\eqref{eqn:assum2}), we have
\[
\left(\frac{1-b(B,\lambda)}{1+a(B,\lambda)}\right)^4\le\frac{\alpha(v_{\lambda})}{\alpha(\Phi_{\lambda})}\le\left(\frac{1+b(B,\lambda)}{1-a(B,\lambda)}\right)^4.
\]

\section{Proof of Theorem~\ref{thm:asymptotic} (asymptotic bounds on $\alpha(\phi_{\lambda})$)}\label{sec:asymptotic}
The first part of Theorem~\ref{thm:asymptotic} follows from the fact that
\[
\frac{1\mp b(B,\lambda)}{1\pm a(B,\lambda)}=1+O(\lambda^{-1/2}),\quad\lambda\rightarrow\infty,
\]
together with~\eqref{eqn:bvc} and the estimates in Theorem~\ref{thm:non-asymptotic}.

Next, if $Q\in{\rm BV}([0,B])$ then we can use the asymptotic expansion of $v_{\lambda}$ from~\eqref{eqn:SV} given by Eq.~(3.3)$_{2N}$ of Fulton and Pruess~\cite{Fulton-Pruess-1994} with $N=1$, and get
\[
v_{\lambda}(y)/v_{\lambda}'(0)=\lambda^{-1/2}\sin(\lambda^{1/2}y)-\frac{1}{2\lambda}\int_0^yQ(s)ds\cdot\cos(\lambda^{1/2}y)+O(\lambda^{-3/2}),\quad y\in[0,B],
\]
where the remainder $O(\lambda^{-3/2})$ is uniform in $y \in [0, B]$.
This, in turn, implies
\begin{align*}
\frac{\|v_{\lambda}\|_2^4}{v_{\lambda}'(0)^4}&=\lambda^{-2}\frac{B^2}{4}\\&-\lambda^{-5/2}B\left(\frac{1}{4}\sin(2B\lambda^{1/2})+\int_{y=0}^B\int_{s=0}^yQ(s)ds\sin(\lambda^{1/2}y)\cos(\lambda^{1/2}y)dy\right)+O(\lambda^{-3})
\end{align*}
and
\begin{align*}
\frac{\|v_{\lambda}\|_4^4}{v_{\lambda}'(0)^4}&=\lambda^{-2}\frac{3B}{8}\\&+\lambda^{-5/2}\Biggl(\frac{\sin(4B\lambda^{1/2})-8\sin(2B\lambda^{1/2})}{32}\Biggr.\\&\Biggl.\qquad-2\int_{y=0}^B\int_{s=0}^yQ(s)ds\sin^3(\lambda^{1/2}y)\cos(\lambda^{1/2}y)dy\Biggr)+O(\lambda^{-3})
\end{align*}
as $\lambda\rightarrow\infty$. Thus $c_-\le\lambda^2\|v_{\lambda}\|_2^4/v_{\lambda}'(0)^4\le c_+$
with
\[
c_{\pm}=\frac{B^2}{4}\pm\lambda^{-1/2}B\left(\frac{1}{4}+B\|Q\|_1\right)+O(\lambda^{-1}),\quad\lambda\rightarrow\infty,
\]
and $d_-\le\lambda^2\|v_{\lambda}\|_4^4/v_{\lambda}'(0)^4\le d_+$ with
\[
d_{\pm}=\frac{3B}{8}\pm\lambda^{-1/2}\left(\frac{9}{32}+2B\|Q\|_1\right)+O(\lambda^{-1}),\quad\lambda\rightarrow\infty.
\]
Finally, if $Q\in{\rm AC}([0,B])$ and $Q'\in{\rm BV}([0,B])$ then we can use the asymptotic expansion of $v_{\lambda}$ given by~\cite[Eq.~(3.3)$_{2N+1}$]{Fulton-Pruess-1994} with $N=1$, to get
\begin{align}\label{eqn:vvv}
v_{\lambda}(y)/v'_{\lambda}(0)&=\lambda^{-1/2}\sin(\lambda^{1/2}y)-\lambda^{-1}\frac{1}{2}\cos(\lambda^{1/2}y)\int_0^yQ(s)ds\nonumber\\&+\lambda^{-3/2}\frac{1}{4}\sin(\lambda^{1/2}y)\left(\int_0^yQ(s)\int_0^sQ(\tau)d\tau ds+Q(0)+Q(y)\right)+O(\lambda^{-2}),
\end{align}
where the remainder $O(\lambda^{-2})$ is uniform in $y$. This, in turn, implies
\begin{align*}
\frac{\|v_{\lambda}\|_2^4}{v'_{\lambda}(0)^4}&=\lambda^{-2}\frac{B^2}{4}\\&-\lambda^{-5/2}B\left(\frac{\sin(2B\lambda^{1/2})}{4}+\int_0^B\sin(\lambda^{1/2}y)\cos(\lambda^{1/2}y)\int_0^yQ(s)dsdy\right)\\&+\lambda^{-3}B\Biggl(\frac{1}{4}\int_0^B\cos^2(\lambda^{1/2}y)\left(\int_0^yQ(s)ds\right)^2dy\Biggr.\\&\Biggl.+\frac{1}{2}\int_0^B\sin^2(\lambda^{1/2}y)\left(\int_0^yQ(s)\int_0^sQ(\tau)d\tau ds+Q(0)+Q(y)\right)dy\Biggr)+O(\lambda^{-7/2})
\end{align*}
and
\begin{align*}
\frac{\|v_{\lambda}\|_4^4}{v'_{\lambda}(0)^4}&=\lambda^{-2}\frac{3B}{8}\\&+\lambda^{-5/2}\left(\frac{\sin(4B\lambda^{1/2})-8\sin(2B\lambda^{1/2})}{32}-2\int_{y=0}^B\sin^3(\lambda^{1/2}y)\cos(\lambda^{1/2}y)\int_0^yQ(s)dsdy\right)\\&+\lambda^{-3}\Biggl(\frac{3}{2}\int_0^B\sin^2(\lambda^{1/2}y)\cos^2(\lambda^{1/2}y))\left(\int_0^yQ(s)ds\right)^2dy\Biggr.\\&\Biggl.+\int_0^B\sin^4(\lambda^{1/2}y)\left(\int_0^yQ(s)\int_0^sQ(\tau)d\tau ds+Q(0)+Q(y)\right)dy\Biggr)+O(\lambda^{-7/2}).
\end{align*}
Now each eigenvalue $\lambda$ is a zero of $\lambda\mapsto v_{\lambda}(B)$~\cite[p. 319, Case 4]{Fulton-Pruess-1994},
and in light of~\eqref{eqn:vvv} we therefore have
\[
\sin(\lambda^{1/2}B)=\lambda^{-1/2}\frac{1}{2}\int_0^BQ(s)ds\cdot\cos(\lambda^{1/2}B)+O(\lambda^{-3/2}),
\]
\[
\sin^2(\lambda^{1/2}B)=\lambda^{-1}\frac{1}{4}\left(\int_0^BQ(s)ds\right)^2\cos^2(\lambda^{1/2}B)+O(\lambda^{-2}),\quad\cos^2(\lambda^{1/2}B)=1+O(\lambda^{-1}),
\]
\[
\sin(2\lambda^{1/2}B)=\lambda^{-1/2}\int_0^BQ(s)ds\cdot\cos^2(\lambda^{1/2}B)+O(\lambda^{-3/2}),
\]
and
\[
\sin(4\lambda^{1/2}B)=\lambda^{-1/2}2\int_0^BQ(s)ds\cdot\cos(2\lambda^{1/2}B)\cos^2(\lambda^{1/2}B)+O(\lambda^{-3/2}).
\]
Also, using integration by parts, we find for any $\phi,\psi\in C^4([0,B])$ with $\phi(0)=0$ that
\begin{align*}
\int_0^B\sin(\lambda^{1/2}y)\cos(\lambda^{1/2}y)\phi(y)dy&=\lambda^{-1/2}\left(\frac{\sin^2(\lambda^{1/2}B)}{2}-\frac{1}{4}\right)\phi(B)+O(\lambda^{-1})\\&=-\lambda^{-1/2}\frac{1}{4}\phi(B)+O(\lambda^{-1}),
\end{align*}
\begin{align*}
\int_0^B\cos^2(\lambda^{1/2}y)\phi(y)dy&=\frac{1}{2}\int_0^B\phi(y)dy+\lambda^{-1/2}\frac{\sin(2\lambda^{1/2}B)}{4}\phi(B)+O(\lambda^{-1})\\&=\frac{1}{2}\int_0^B\phi(y)dy+O(\lambda^{-1}),
\end{align*}
\begin{align*}
\int_0^B\sin^2(\lambda^{1/2}y)\psi(y)dy&=\frac{1}{2}\int_0^B\psi(y)dy-\lambda^{-1/2}\frac{\sin(2\lambda^{1/2}B)}{4}\psi(B)+O(\lambda^{-1})\\&=\frac{1}{2}\int_0^B\psi(y)dy+O(\lambda^{-1}),
\end{align*}
\begin{align*}
\int_0^B\sin^3(\lambda^{1/2}y)\cos(\lambda^{1/2}y)\phi(y)dy&=\lambda^{-1/2}\left(\frac{\sin^4(\lambda^{1/2}B)}{4}-\frac{3}{32}\right)\phi(B)+O(\lambda^{-1})\\&=-\lambda^{-1/2}\frac{3}{32}\phi(B)+O(\lambda^{-1}),
\end{align*}
\begin{align*}
\int_0^B\sin^2(\lambda^{1/2}y)\cos^2(\lambda^{1/2}y)\phi(y)dy&=\frac{1}{8}\int_0^B\phi(y)dy\\&+\lambda^{-1/2}\Biggl(\frac{\sin(2\lambda^{1/2}B)-4\sin(\lambda^{1/2}B)\cos^3(\lambda^{1/2}B)}{16}\phi(B)\Biggr.\\&\Biggl.-\frac{\cos^4(\lambda^{1/2}B)}{2}\phi'(B)+\frac{3}{16}\phi'(B)+\frac{5}{16}\phi'(0)\Biggr)+O(\lambda^{-1})\\&=\frac{1}{8}\int_0^B\phi(y)dy+\lambda^{-1/2}\frac{5}{16}\left(\phi'(0)-\phi'(B)\right)+O(\lambda^{-1}),
\end{align*}
and
\begin{align*}
\int_0^B\sin^4(\lambda^{1/2}y)\psi(y)dy&=\frac{3}{8}\int_0^B\psi(y)dy\\&-\lambda^{-1/2}\left(\frac{\sin^3(\lambda^{1/2}B)\cos(\lambda^{1/2}B)}{4}+\frac{3\sin(2\lambda^{1/2}B)}{16}\right)\psi(B)+O(\lambda^{-1})\\&=\frac{3}{8}\int_0^B\psi(y)dy+O(\lambda^{-1}).
\end{align*}
Using these expansions, we get
\begin{align*}
\lambda^2\frac{\|v_{\lambda}\|_2^4}{v'_{\lambda}(0)^4}&=\frac{B^2}{4}+\lambda^{-1}\frac{B}{4}\Biggl[\frac{1}{2}\int_0^B\left(\int_0^yQ(s)ds\right)^2dy\Biggr.\\&\Biggl.+\int_0^B\left(\int_0^yQ(s)\int_0^sQ(\tau)d\tau ds+Q(0)+Q(y)\right)dy\Biggr]+O(\lambda^{-3/2})
\end{align*}
and
\begin{align*}
\lambda^2\frac{\|v_{\lambda}\|_4^4}{v'_{\lambda}(0)^4}&=\frac{3B}{8}\\&+\lambda^{-1}\frac{3}{8}\Biggl[\frac{1}{2}\int_0^B\left(\int_0^yQ(s)ds\right)^2dy\Biggr.\\&\Biggl.+\int_0^B\left(\int_0^yQ(s)\int_0^sQ(\tau)d\tau ds+Q(0)+Q(y)\right) dy\Biggr]+O(\lambda^{-3/2}).
\end{align*}
Note that the factors multiplying $\lambda^0$ and $\lambda^{-1}$ in $\|v_{\lambda}\|_2^4$ are proportional to those in $\|v_{\lambda}\|_4^4$, with the proportionality constant $2B/3$. We thus have
\begin{align*}
\alpha(v_{\lambda})=\frac{B^2/4+\lambda^{-1}s+O(\lambda^{-3/2})}{(3/2B)(B^2/4+\lambda^{-1}s+O(\lambda^{-3/2}))}=\frac{2B}{3}+O(\lambda^{-3/2}),
\end{align*}
where
\[
s=\frac{B}{4}\Biggl[\frac{1}{2}\int_0^B\left(\int_0^yQ(s)ds\right)^2dy+\int_0^B\left(\int_0^yQ(s)\int_0^sQ(\tau)d\tau ds+Q(0)+Q(y)\right) dy\Biggr].
\]

\section{Proof of Proposition~\ref{prop:landscape-function} }\label{sec:proof2}
By construction $T$ is self-adjoint and  diagonalizable. Hence it can be written in the following form
\begin{equation} \label{eigenexpansion}
    T = \sum_{j=1}^\infty \lambda_j P_j,
\end{equation}    
where $(\lambda_j)_{j\in \bm{N}_0}$ is the strictly increasing sequence of  eigenvalues of $T$, and  $P_j$ are the orthogonal projections onto the eigenspaces associated to $\lambda_j, \; j\in \bm{N}_0.$\\

Since $\ell_k \in D(T)$, it has the following expansion
\[
\ell_k= \tilde \ell_k/\|\tilde \ell_k\|_{2},\quad \tilde \ell_k= \sum_{j=1}^\infty \lambda_j^{-k} P_j1.
\]
Straightforward computations give
\begin{eqnarray}\label{intermediate estimate1}
\|\lambda_1^k\tilde \ell_k -P_11 \|_2 \leq L^{1/2} \left(\frac{\lambda_1}{\lambda_2}\right)^k, \quad \textrm{and  } \|P_11 \|_2 \leq \|\lambda_1^k\tilde \ell_k\|_2.
\end{eqnarray}
We then deduce
\begin{equation} \label{intermediate estimate2}
\|\lambda_1^k \tilde \ell_k\|_2 - \|P_11\|_2 \leq L^{1/2}\left(\frac{\lambda_1}{\lambda_2}\right)^k.
\end{equation}
Similarly, since $T\tilde \ell_k \in L^2(]0,L[)$, we have 
\begin{eqnarray}\label{intermediate3}
\|\lambda_1^kT^{1/2}\tilde \ell_k - T^{1/2}P_11\|_2 \leq (\lambda_1L)^{1/2} \left(\frac{\lambda_1}{\lambda_2}\right)^{k-1/2}. 
\end{eqnarray}
Recall that $\phi_1>0$, which implies $ \|P_11\|_2>0 $ and $\phi_1= P_11/\|P_11\|_2 $.
Now combining inequalities \eqref{intermediate estimate1},   \eqref{intermediate estimate2}  and  \eqref{intermediate3}, we get  
\begin{eqnarray} 
\|T^{1/2}\ell_k - T^{1/2}\phi_1\|_2 \leq \|\lambda_1^kT^{1/2}\tilde \ell_k - T^{1/2}P_11\|_2  \|P_11\|_2^{-1}+
\lambda_1^{1/2} \left(\|\lambda_1^k \tilde \ell_k\|_2 - \|P_11\|_2 \right)  \|P_11\|_2^{-1}\nonumber\\
\leq 2(\lambda_1L)^{1/2} \left(\frac{\lambda_1}{\lambda_2}\right)^{k-1/2}  \|P_11\|_2^{-1}.\label{intermediate result one}
\end{eqnarray}
On the other hand, we have 
\begin{equation*}
|\varphi(x)| \leq \int_0^x|\varphi^\prime(s)|ds\leq L^{1/2}\|\varphi^\prime\|_2, \quad \forall x\in ]0, L[,
\end{equation*}
for all $\varphi \in C^\infty_0(]0,L[)$.\\

Since  $C^\infty_0(]0,L[)$ is dense in $D(T^{1/2})$, we get
\begin{equation} \label{Sobolev embedding}
\|\varphi\|_\infty \leq L^{1/2} \|T^{1/2}\varphi\|_2, \quad \forall    \varphi \in D(T^{1/2}).
\end{equation}
Combining inequalities \eqref{Sobolev embedding} and \eqref{intermediate result one}, we obtain the desired result. 
\section{Proof of Proposition \ref{prop: iterated landscape function} } \label{sec:proof2bis}
Using the spectral expansion \eqref{eigenexpansion}, we get
\begin{equation*}
\ell_{k,t} = \sum_{j=1}^\infty \frac{1}{(t\lambda_j)^k} P_j1. 
\end{equation*}
Hence 
\begin{equation*}
T^{1/2}\left(  \ell_{k,t}-\sum_{j=1}^{n_0} \frac{1}{(t\lambda_j)^k} P_j1\right) =  \sum_{j=n_0+1}^\infty \frac{1}{(t\lambda_j)^k} T^{1/2}P_j1.
\end{equation*}
Therefore
\begin{equation}
\left\| T^{1/2}\left(  \ell_{k,t}-\sum_{j=1}^{n_0} \frac{1}{(t\lambda_j)^k} P_j1\right) \right\|_2 \leq \frac{L^{1/2}}{t^{1/2}}
\frac{1}{(t\lambda_{n_0+1})^{k-1/2}}.
\end{equation}
Applying again the Sobolev inequality~\eqref{Sobolev embedding}, we recover the final estimate.

\section{The landscape function: numerical tests} \label{sec:numerical tests} 
We start by illustrating the consequences of Proposition~\ref{prop:landscape-function}. For Figure~\ref{fig:Prop1} we use $L=1$ and
\[
p(x)=\tanh(40 x/L - 10)+1.1,\quad q(x)=0,\quad x\in[0,L],
\]
as in~\eqref{eqn:p} in Section~\ref{sec:intro}), while Figure~\ref{fig:setup2} shows the graphs of
\[
p(x)=\tanh(40 x/L - 20)+1.1,\quad q(x)=2+\sin(2\pi x),\quad x\in[0,L],
\]
used, with $L=1$, for the results of Figure~\ref{fig:Prop1--2}. Finally, Figure~\ref{fig:setup3} shows the functions
\[
p(x)=\tanh(40 x/L - 10)+1.1,\quad q(x)=2+\sin(2\pi x),\quad x\in[0,L],
\]
used, with $L=5$, for the results of Figure~\ref{fig:Prop1--3}.
\begin{figure}
    \centering
    \includegraphics[scale=0.25]{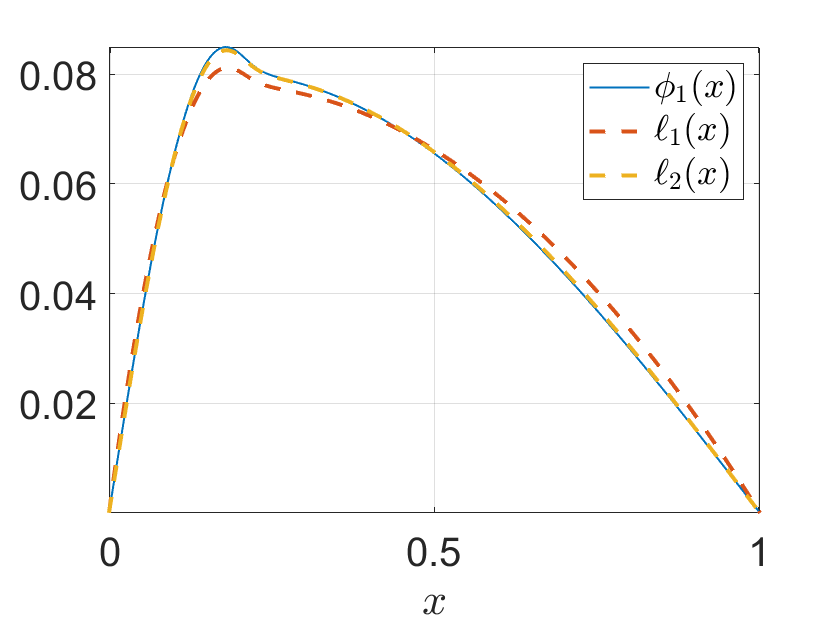}
    \includegraphics[scale=0.25]{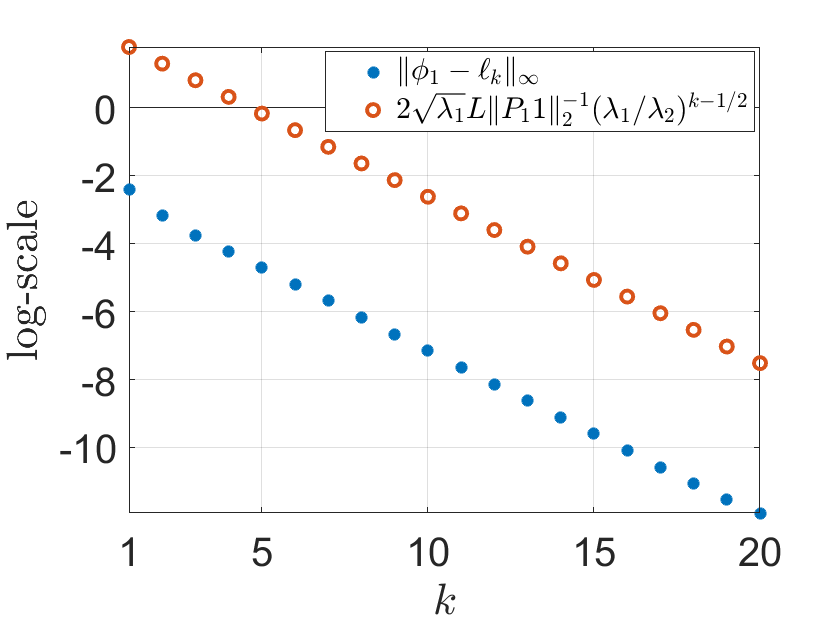}
    \caption{Left: the functions $\ell_k$ approach the first eigenvector $\phi_1$ pointwise as $k$ increases. Right: actual value vs. upper bound on $\|\phi_1-\ell_k\|_{\infty}$, see Proposition~\ref{prop:landscape-function}.}
    \label{fig:Prop1}
\end{figure}
\begin{figure}
    \centering
    \includegraphics[scale=0.3]{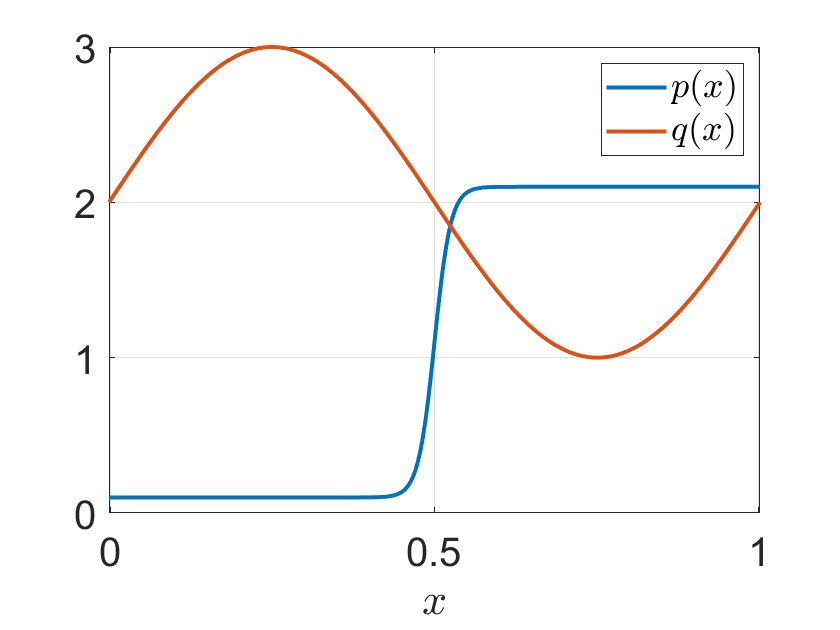}
    \caption{The functions $p(x)$ and $q(x)$ used for the results of Figure~\ref{fig:Prop1--2}.}
    \label{fig:setup2}
\end{figure}
\begin{figure}
    \centering
    \includegraphics[scale=0.25]{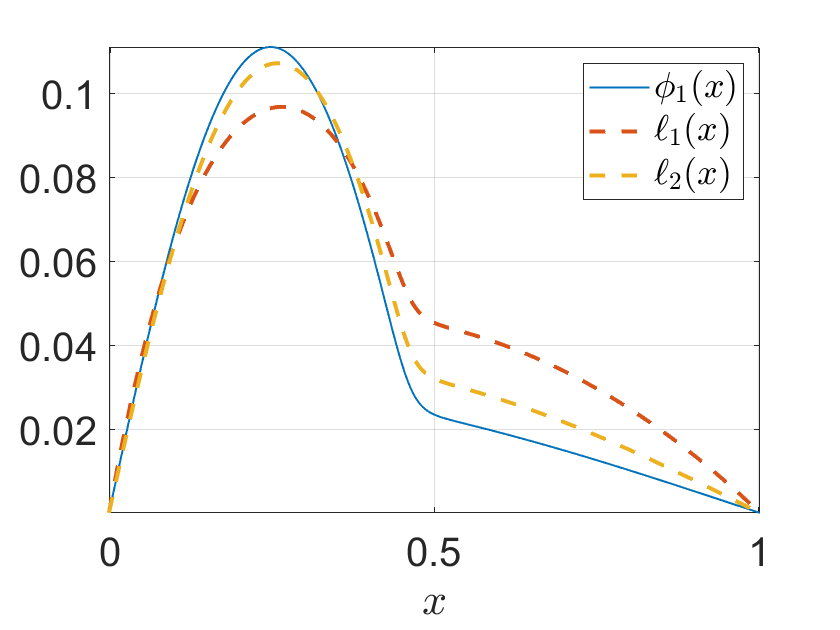}
    \includegraphics[scale=0.25]{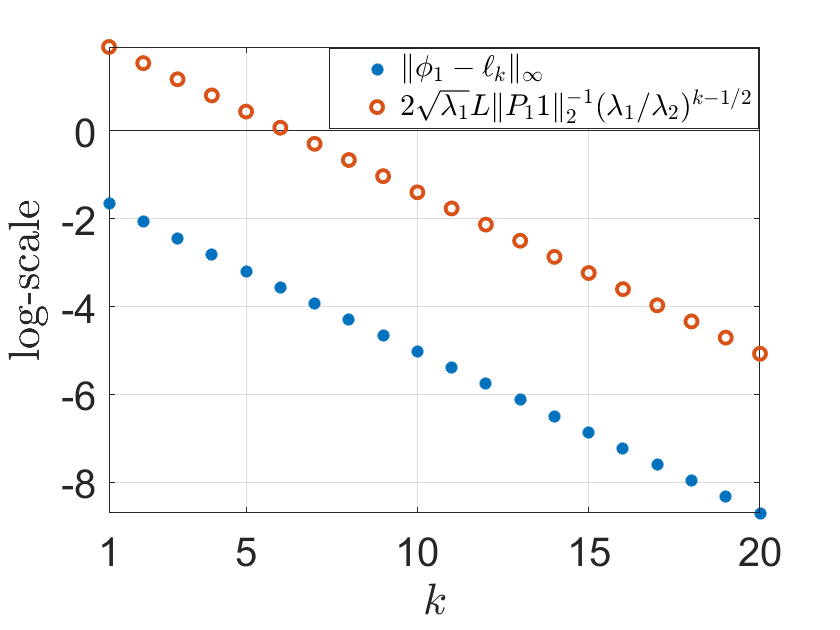}
    \caption{Left: the functions $\ell_k$ approach the first eigenvector $\phi_1$ pointwise as $k$ increases. Right: actual value vs. upper bound on $\|\phi_1-\ell_k\|_{\infty}$, see Proposition~\ref{prop:landscape-function}.}
    \label{fig:Prop1--2}
\end{figure}
\begin{figure}
    \centering
    \includegraphics[scale=0.3]{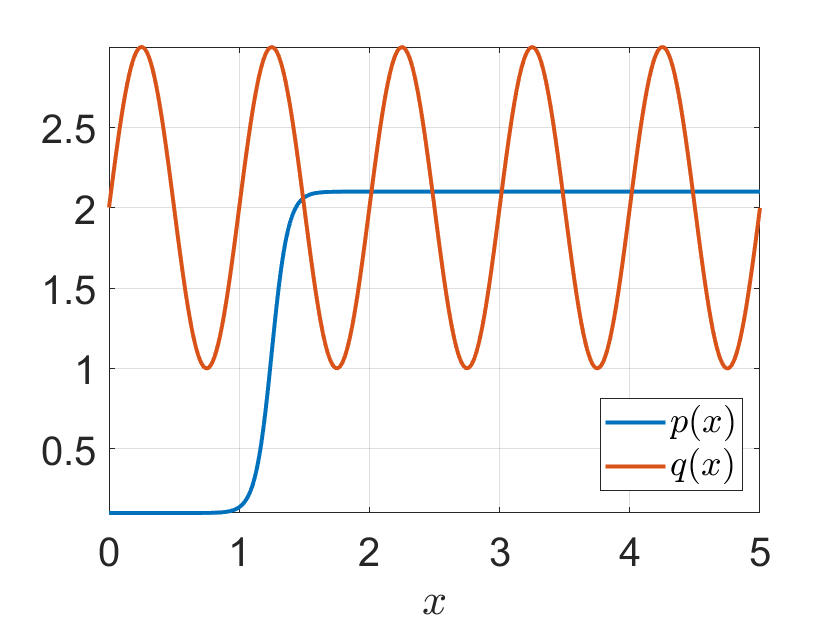}
    \caption{The functions $p(x)$ and $q(x)$ used for the results of Figure~\ref{fig:Prop1--3}.}
    \label{fig:setup3}
\end{figure}
\begin{figure}
    \centering
    \includegraphics[scale=0.25]{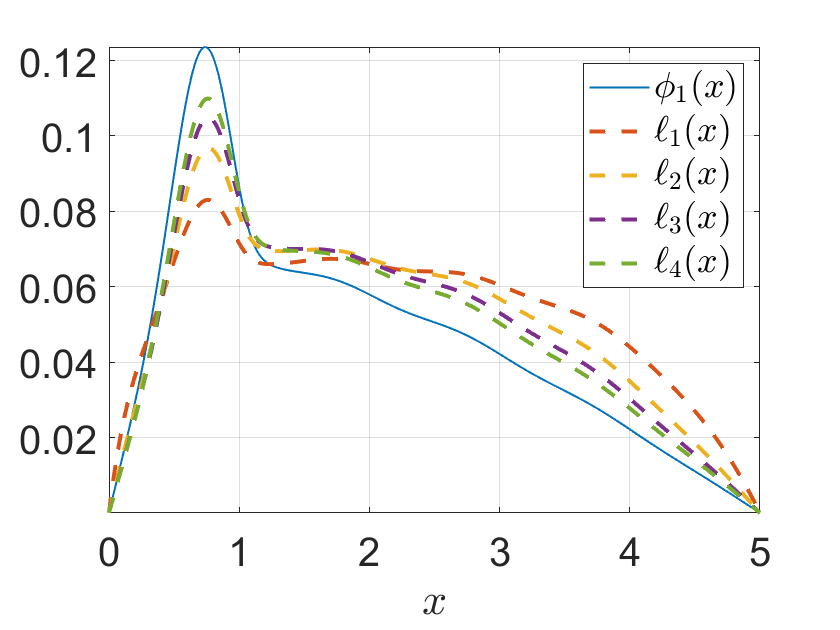}\includegraphics[scale=0.25]{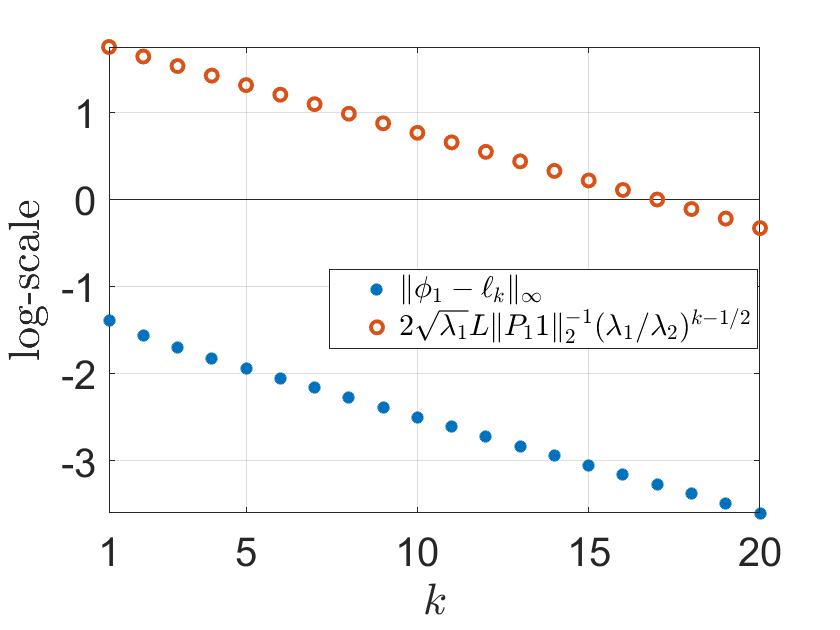}
    \caption{Left: the functions $\ell_k$ approach the first eigenvector $\phi_1$ pointwise as $k$ increases. Right: actual value vs. upper bound on $\|\phi_1-\ell_k\|_{\infty}$, see Proposition~\ref{prop:landscape-function}.}
    \label{fig:Prop1--3}
\end{figure}
Next, in Figure~\ref{fig:Prop2} we illustrate the validity of the upper bound on $\|\ell_{k,t}-\sum_{j=1}^{n_0}(t\lambda_j)^{-k}P_j1\|_{\infty}$, as given in Proposition~\ref{prop: iterated landscape function}. Since the constants $(t\lambda_j)^{-1}$, $j=1,\dots,n_0$, are greater than 1, the numerical error present in the above $L^{\infty}$-norm can grow exponentially with $k$. To avoid this numerical instability, in Figure~\ref{fig:Prop2} we plot the equivalent quantity $\|\sum_{j=n_0+1}^{\infty}(t\lambda_j)^{-k}(t\lambda_j)^{-k}P_j1\|_{\infty}$, with the series truncated at $j=20$.
\begin{figure}
    \includegraphics[scale=0.25]{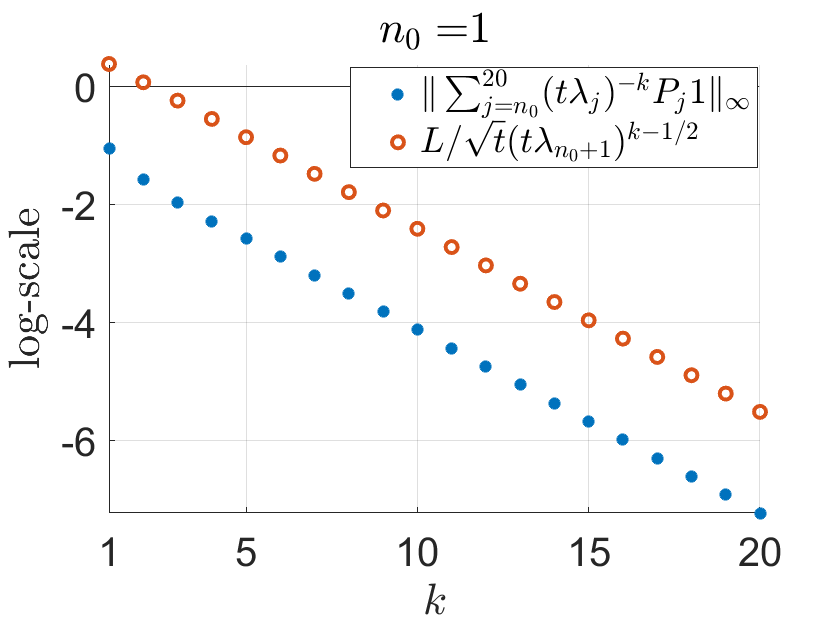}    \includegraphics[scale=0.25]{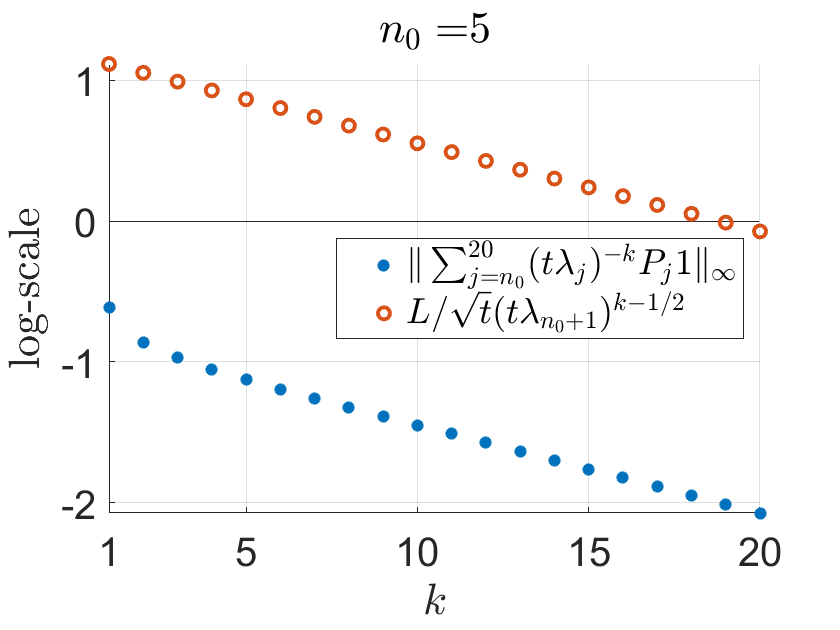}    \\
    \includegraphics[scale=0.25]{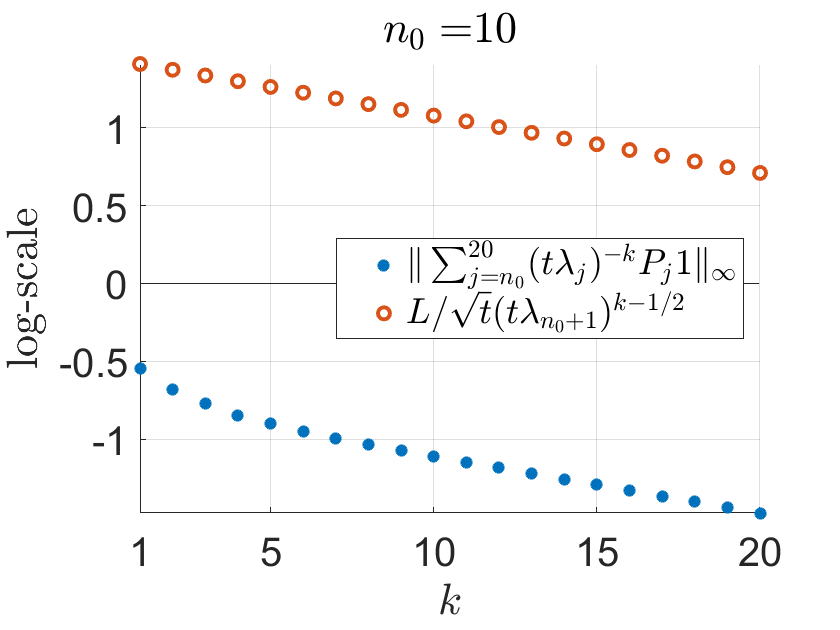}    \includegraphics[scale=0.25]{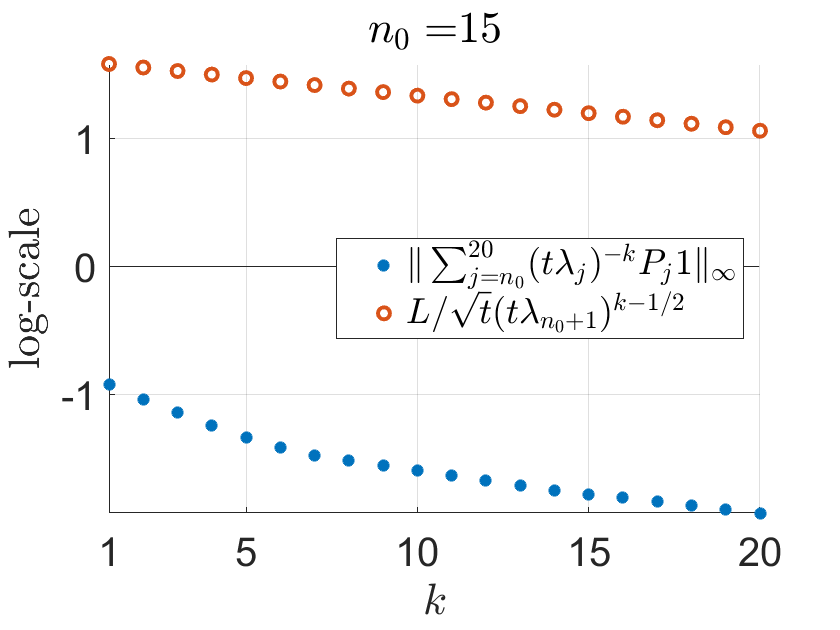}
    \caption{Illustration of the upper bound of Proposition~\ref{prop: iterated landscape function} for different values of $n_0$.}
    \label{fig:Prop2}
\end{figure}

\section*{Acknowledgments}
\thanks{M. K. was supported by The Villum Foundation (grant no. 25893).}

\bibliographystyle{amsplain}
\bibliography{biblio}

\end{document}